\documentclass[11pt]{article}

\usepackage{amsmath,amsthm,amssymb,color}
\usepackage{graphicx}
\usepackage{tikz-cd}
\usepackage{amscd}
\usepackage[normalem]{ulem}
\usepackage{xcolor}
\newcommand\nsandor{\bgroup\markoverwith{\textcolor{blue}{\rule[0.5ex]{2pt}{0.4pt}}}\ULon}

\parskip=\medskipamount
\def\clift#1{#1^{\scriptscriptstyle{\mathrm{C}}}}

\def\vlift#1{#1^{\scriptscriptstyle{\mathrm{V}}}}

\def\fpd#1#2{{\displaystyle\frac{\partial #1}{\partial #2}}}
\def\spd#1#2#3{{\displaystyle\frac{\partial^2 #1}
{\partial #2\partial #3}}}

\def\clift#1{#1^{\scriptscriptstyle{\mathrm{C}}}}

\def\vlift#1{#1^{\scriptscriptstyle{\mathrm{V}}}}

\def\E{\mathcal{E}}
\def\R{\mathbb{R}}

\def\onehalf{{\textstyle\frac12}}

\def\la{{\mathfrak g}}

\def\vectorfields#1{\mathcal{X}(#1)}

\def\tr{\mathop{\mathrm{tr}}}
\def\Ad{\mathop{\mathrm{ad}}\nolimits}
\def\ad{\Ad}

\def\h{{\mathbf h}}
\def\F{{\mathbf F}}
\def\E{{\mathbf E}}
\def\BL{{\mathbf L}}
\def\H{{\mathbf H}}

\theoremstyle{plain}

\newtheorem{lemma}{Lemma}
\newtheorem{proposition}{Proposition}
\newtheorem{definition}{Definition}

\begin{document}

\title{Homogeneous nonlinear splittings and Finsler submersions}

\author{ S.\ Hajd\'u and T.\ Mestdag \\[2mm]
	{\small Department of Mathematics,  University of Antwerp,}\\
	{\small Middelheimlaan 1, 2020 Antwerpen, Belgium}
}

\date{}

\maketitle

\begin{abstract}
A nonlinear splitting on a fibre bundle is a generalization of an Ehresmann connection. An example is given by the  homogeneous nonlinear splitting  of a  Finsler function on the total manifold of a fibre bundle. We show how homogeneous nonlinear splittings and nonlinear lifts can be used to construct submersions between Euclidean, Minkowski and Finsler spaces. As an application we consider a semisimple Lie algebra and use our methods to give new examples of Finsler functions on a reductive homogeneous space.
\vspace{3mm}

\textbf{Keywords:} Homogeneous nonlinear splittings, nonlinear lifts, Minkowski and Finsler submersions,  bi-invariant Finsler functions on Lie groups, reductive homogeneous Finsler spaces.

\vspace{3mm}

\textbf{Mathematics Subject Classification:} 
53B40, 
53C60, 
53C05, 
53C30. 
\end{abstract}

\section{Introduction}

Structure preserving submersions between two manifolds play  a fundamental role in many parts of differential geometry, such as Lie group theory and Riemannian geometry. The main purposes for investigating such submersions can be summarized in two not unrelated manners: either to compare the geometry of the two manifolds, or to project a version of the geometric structure of one manifold onto the other. In particular, Riemannian submersions have been extensively studied since their inception (see e.g.\ the recent books \cite{falcitelli,sahin} for surveys, and the last chapter of \cite{falcitelli} for applications to physics). 

In \cite{Alv}, the underlying geometry of the manifolds is coming from  Finsler functions. The notion of a submersion between two Finsler manifolds extends that of a Riemannian submersion, and it can be used to construct new, interesting examples of Finsler manifolds from old ones. The notion of a Finsler submersion has received much attention in the last decades, see e.g.\
 \cite{paper1,libing,deng2,xu,xu2} for a non-exhaustive list of results and applications on Finsler submersions.

In our recent paper \cite{ourthirdarticle} we have introduced  the notion of a nonlinear splitting on a fibre bundle. This concept is a generalization of an Ehresmann connection on a fibre bundle, when one leaves a certain linearity requirement out of the definition. In accordance with the structure of the fibre bundle, we may speak of linear, affine, principal or homogeneous nonlinear splittings. Our main motivation for introducing nonlinear splittings comes from diverse topics in geometric mechanics, such as symmetry-group reduction of Lagrangian systems, mechanical systems with nonholonomic constraints, or Lagrangian systems that are exposed to  magnetic forces. The main goal of the current contribution is to show how homogeneous nonlinear splittings lie at the basis of many results concerning Finsler submersions.

In Section~\ref{secnonlinsplit} we recall the notion of a homogeneous nonlinear splitting on a fibre bundle and we discuss its basic properties. In particular, we show that a Finsler function on a fibre bundle defines a homogeneous nonlinear splitting and we investigate when the canonical spray of a Finsler manifold is tangent to the horizontal submanifold of the nonlinear splitting. In particular, we show that this tangency property is a projective invariant of the projective class of  sprays of a Finsler function.

Each tangent space of a Finsler manifold carries a so-called Minkowski norm.
In Section~\ref{minkowskinormssection} we consider Minkowski vector spaces,  instead of Finsler manifolds, and we recall the notion of a submersion between Minkowski normed spaces. This notion can be thought of as the infinitesimal version of a Finsler submersion. In Proposition~\ref{basicprop} we give two equivalent characterizations of this concept. This extends a result of \cite{Alv}, where only one direction was stated (without proof). As an example, we investigate how Euclidean submersions fit as examples of Minkowski submersions.

 In Section~\ref{nonlinearliftssection} we introduce the notion of a nonlinear lift between two vector spaces. Given a linear surjective map between two vector spaces, and given a Minkowski norm on the first vector space, we present in Proposition~\ref{subduced} the construction of a so-called subduced Minkowski norm on the second vector space. Our goal is to show that a nonlinear lift lies at the heart of this construction.
Each Minkowski space can be endowed with a family of Euclidean inner products. We show in Proposition~\ref{Euclprop} how they fit within the context of Minkowski submersions, and we focus, again, on their role of the associated nonlinear lifts. We indicate that the formulation of a similar result (in a less general context) of \cite{libing} was stated somewhat imprecise, and that the nonlinear lift can help to correct it.

In Section~\ref{Sec5} we recall from  \cite{Alv} the definition of a Finsler submersion. We show in Proposition~\ref{finslersplitting} how the nonlinear splitting, associated to a Finsler function on the total manifold on a fibre bundle can be employed in the construction of a subduced Finsler function on the base manifold. A sufficient condition for our construction is exactly that the canonical spray of the Finsler function is tangent to the horizontal submanifold of the nonlinear splitting (as we had investigated in Section~\ref{secnonlinsplit}). As an application we discuss reductive homogeneous Finsler spaces in Section~\ref{homspacesec}. We point out in Proposition~\ref{mainn} that
a Finsler function on the Lie group $G$, that is both left $G$- and right $H$-invariant, is sufficient to guarantee the existence of a nonlinear lift and of a subduced Finsler function on the homogeneous manifold $G/H$. 

Bi-invariant Finsler functions on Lie groups have been extensively investigated in the literature (see e.g.\ \cite{latifi,deng1}). However, explicit examples for such functions - that are non-Riemannian - are rarely calculated. Since bi-invariant functions satisfy the conditions of Proposition~\ref{mainn} we can use them (in Section~\ref{exsec}) to provide some examples of Finsler functions on homogeneous spaces. In particular, we consider $S^3$ as the homogeneous space $SO(4)/SO(3)$. 
Due to the one-to-one correspondence between Ad-invariant functions on a semisimple Lie algebra and its dual (Proposition~\ref{dual}), and thanks to an explicit calculation of such invariants in \cite{patera} for $SO(4)$, we can use our constructive method to work out examples of nolinear splittings and  Finsler functions on the  homogeneous  space $S^3$.

\section{Homogeneous nonlinear splittings} \label{secnonlinsplit}
Throughout this paper, we will work with a fibre bundle $\pi: M\rightarrow N$, and we will denote by $\tau:TM\rightarrow M$ and $\bar\tau:TN\rightarrow N$ the tangent bundles of $M$ and $N$, respectively. The natural projections of the pullback bundle $\pi^*TN$ of $\bar\tau$ by $\pi$, 
\[
\pi^*TN=\{(m,v_n)\in M\times TN~|~\pi(m)=\bar\tau(v_n)\},
\]
will be denoted by $p_1:\pi^*TN \to M$ and $p_2:\pi^*TN \to TN$.  

In the short exact sequence 	\[
\begin{tikzcd}
0\arrow{r} & V\pi\arrow{r}{} & TM\arrow{r}{(\tau,T\pi)} & \pi^*TN\arrow{r} & 0,
\end{tikzcd}
\]
$V\pi$ stands for the vertical bundle ${\rm Ker}\, T\pi$ of $\pi$ and $(\tau,T\pi)$
stands for the linear bundle map $TM \to \pi^*TN, (w_m) \mapsto (m, T\pi(w_m))$. From \cite{ourthirdarticle} we recall the  definition of a nonlinear splitting.

\begin{definition}
	A nonlinear splitting on $\pi:M\rightarrow N$
	is a map $h:\pi^*TN\rightarrow TM$ which is
	 smooth on the slit pullback bundle $\pi^*\mathring{T}N$,
	 fibre-preserving (i.e.\ $\tau \circ h = p_1$) and satisfies $T \pi \circ h = p_2$. The set   ${\rm Im}\,h\subset TM$ is the horizontal manifold of $h$.

A nonlinear splitting  is  homogeneous  if
	\[
	h(m,\lambda v_n)=	\lambda h(m,v_n), \qquad \forall \lambda \in \R^+,
	\]
that is, if $h$ is positive homogeneous of degree 1.
	\end{definition}

In this definition, $\mathring{T}N$ stands for the tangent manifold $TN$ from which the zero section has been removed.
Let $(x^i)$ be coordinates  on $N$ and $(q^a)=(x^i,y^{\alpha})$ coordinates on $M$ that are adjusted to the fibre bundle structure of $\pi$. We will denote the corresponding natural fibre coordinates on $TM$ by $(u^a)=(v^i,w^{\alpha})$. A local expression of a nonlinear splitting $h$ is then
\[
h(x^i,y^{\alpha},v^i)=(x^i,y^{\alpha},v^i,h^{\alpha}(x,y,v)).
\]
We call the functions $h^{\alpha}$  the coefficients of $h$. When the nonlinear splitting is homogeneous these coefficients $h^\alpha$ are $1^+$-homogeneous functions.
Euler's theorem (see e.g.\  \cite{Szilasi}) states that this is equivalent with the property
\begin{equation}\label{eulerstheorem}
	v^i\fpd{h^\alpha}{v^i} = h^\alpha.
\end{equation}
A nonlinear splitting is an {\em Ehresmann connection} on the fibre bundle $\pi: M\to N$ when its coefficients $h^\alpha$ are linear functions in $v^i$.

Our main example for a homogeneous nonlinear splitting comes from Finsler geometry. Let $M$ be a manifold.
A {\em second-order ordinary differential equation field} $S$ on $M$  is a vector field on $TM$ with the property that all its integral curves $\gamma(t)$ in $TM$ are lifted curves $\dot c(t)$ of curves $c(t)$ on $M$. Locally, $S$ can be expressed as
\[
S={u}^a\fpd{}{q^a}+f^a\fpd{}{u^a}.
\]
A second-order vector field $S$ is a {\em spray}, if it satisfies
\[
[\Delta,S]=S.
\]
Herein is  $\Delta=u^a\fpd{}{u^a}$ the Liouville vector field. In case of a spray, the coefficients $f^a$ are positive homogeneous functions of degree 2.

\begin{definition}
	A Finsler function $F:TM\rightarrow \R$ is a non-negative continuous function on the tangent bundle $TM$, which is
	  positive homogeneous,   smooth on the slit tangent bundle $\mathring{T}M$
	and strongly convex.
\end{definition} 
The function $E:=\frac{1}{2}F^2$ is called the energy of the Finsler function. In  local coordinates $(q^a,u^a)$ on $TM$, strong convexity means that the matrix
\[
\left(\frac{\partial^2 E}{\partial u^{a}\partial u^{b}}\right)
\]
is positive definite everywhere. If $\clift{X} = X^a \partial/\partial q^a + (\partial X^b/ \partial q^a) \partial/\partial u^a$ and $\vlift{X} = X^a\partial/\partial u^a$ stand for, respectively, the complete lift and the vertical lift of a vector field $X = X^a \partial/\partial q^a$ on $M$, then the relation
	\begin{equation}\label{ELsode}
	S  (\vlift{X}(E)) - \clift{X}(E) = 0,\qquad \forall X \in \vectorfields{M},
	\end{equation}
determines a unique spray  $S $ on $M$. This spray is called  {\em the canonical spray of the Finsler function $F$}. 

A spray $\tilde{S}$ is said to be {\em projectively related to $S$} if there exists a (1-homogeneous) function $P$ on $TM$ such that $\tilde S=S-2P\Delta$.
The geometric interpretation of this property is that $\tilde S$ and $S$ share the same base integral curves, when considered as point sets. The set of all sprays that are projectively related to each other is called a {\em projective class of sprays}. Properties of sprays, that are invariant under projective changes and therefore `geometric', are referred to as `projective invariants'.

The base integral curves of the canonical spray represent geodesics that are parametrized by arc length. In case of the canonical spray $S$ of a Finsler function $F$, there exists an easy characterization of its projective class of sprays (see e.g.\ \cite{CMS}): The class consists of all sprays $\tilde S$ that satisfy the property 
\[	\tilde S (\vlift{X}(F)) - \clift{X}(F) = 0,\qquad \forall X \in \vectorfields{M}.
\]

We now return to the case of a fibre bundle $\pi: M \to N$ and show that each Finsler function $F$ on $M$ generates a homogeneous nonlinear splitting. From the  strong convexity of $F$ we know  that the submatrix
\[
\left(\frac{\partial^2 E}{\partial w^{\alpha}\partial w^{\beta}}\right)
\]
is everywhere non-singular. For this reason, the implicit function theorem guarantees the local existence of a map $h:\pi^*TN\rightarrow TM$, as the solution of
\begin{equation}\label{definingrelation}
\frac{\partial E}{\partial w^{\alpha}}\circ h=0 \qquad\Leftrightarrow\qquad w^\alpha = h^\alpha(x,y,v).
\end{equation}
This defining relation of a nonlinear splitting  can also be written as 
\[
\vlift{Y}(E) \circ h =0, \qquad \forall \mbox{$\pi$-vertical vector field $Y$ on $M$}.
\]
\begin{definition} \label{induced} The  nonlinear splitting on $\pi: M \to N$, induced by the Finsler function $F$ on $M$, is the map $h:\pi^*TN\rightarrow TM$  determined by the relation (\ref{definingrelation}).
\end{definition}

\begin{proposition}\label{symmetryprop} Let $\pi: M \to N$ be a fibre bundle and $F$  a Finsler function on $M$.
	\begin{itemize} \item[(i)] The nonlinear splitting $h$ on $\pi$, induced by  $F$, is a homogeneous splitting. 
	
\item[(ii)]  The canonical  spray $S$ of $F$ is tangent to the horizontal manifold ${\rm Im}\,h$ if and only if
	\begin{equation}\label{symmetrycondition}
	\clift{Y}(E)\circ h=0,	
	\end{equation}
	for any vector field $Y$ on $M$ that is $\pi$-vertical.
	\item[(iii)] If the canonical spray $S$ is tangent to ${\rm Im}\,h$, then so are  all other sprays of its   projective class.
\end{itemize}
\end{proposition}

\begin{proof} 
(i)	The energy $E=\frac{1}{2}F^2$ of the Finsler function $F$ is a $2^+$-homogenous function on $TM$, and satisfies $\Delta(E)=2E$, where $\Delta$ is the Liouville vector field on $M$. Since, for each $X\in\vectorfields{M}$, $[\Delta,\vlift X] = -\vlift X$, it is easy to see that $\Delta(\vlift X(L)) = \vlift X(L)$, from which we may conclude that $\vlift{X}(L)$ is a $1^+$-homogenous function on $TM$.
As a consequence, for a $\pi$-vertical vector field $Y$ on $M$, both
\[ \vlift{Y}(E)(x,y,\lambda v,\lambda h(x,y,v)) =\lambda \vlift{Y}(E)(x,y,  v, h(x,y,v)) =0
\] and
\[ \vlift{Y}(E)(x,y,\lambda v, h(x,y,\lambda v))=0.
\]
	 From the uniqueness in the implicit function theorem, we conclude that $\lambda h^\alpha(x,y,v)=h^\alpha(x,y,\lambda v)$, which expresses the homogeneity of the splitting. 
	
(ii) The relation (\ref{ELsode}), when restricted to ${\rm Im}\,h$ and to $\pi$-vertical vector fields  gives
	 \[
	 S (\vlift{Y}(E))\circ h - \clift{Y}(E)\circ h= 0.
	 \]
	 Under the condition (\ref{symmetrycondition}), this  is equivalent to $S (\vlift{Y}(E))\circ h=0$, which expresses that $S$ is tangent to ${\rm Im}\,h$.

(iii) The difference between $S$ and $\tilde S$ is always a multiple of the Liouville vector field $\Delta$. This vector field,  $\Delta = v^i \partial/\partial {v^i}+ w^a \partial/\partial  {w^a}$,  is always tangent to the horizontal manifold $\mathcal H$ of a homogeneous nonlinear splitting, since \[\Delta(w^\alpha -h^\alpha) = w^\alpha - v^i \fpd{h^\alpha}{v^i} =0,
	\]
	everywhere where  $w^\alpha=h^\alpha$, in view of relation (\ref{eulerstheorem}).
\end{proof}

The property (iii) of Proposition~\ref{symmetryprop} shows that, when satisfied, condition (\ref{symmetrycondition}) is not specific to $S_E$ alone, but also valid for each of the sprays in its projective class.
In Section~\ref{Sec5} we will relate the tangency of $S$ to ${\rm Im}\,h$ to the concept of a Finsler submersion. But, first we need to revise some aspects of Euclidean and Minkowski submersions.

\section{Minkowski norms and Minkowski submersions}\label{minkowskinormssection}

In this section we first recall some preliminaries and results  on Minkowski norms that can be found in e.g.\ \cite{BCS2,Szilasi}.

Recall first  that a function $\F:V\rightarrow \R$ is said to be convex if 
	\[
	\F(tu +(1-t){\tilde u}) \leq t\F(u)+(1-t)\F({\tilde u}),
	\]
	  for all $u,\tilde u\in V$ and $t\in[0,1]$. $\F$ is strictly convex if strict inequality holds for all $u\neq \tilde u$.

\begin{definition}
	A Minkowski norm $\F$ on a vector space $V$ is a non-negative real-valued function with the following properties:
	\begin{itemize}
		\item $\F$ is regular: $\F$ is $C^{\infty}$ on $V\setminus \{0\}$,
		\item $\F$ is positive homogeneous: $\F(\lambda u)=\lambda \F(u)$ for all $u\in V$ and $\lambda >0$,
		\item $\F$ is strongly convex:  $\E := \onehalf \F^2$ is strictly convex, i.e.\  the matrix of functions $(g_{ab})$, with $g_{ab}	(u):= \spd{\E}{u^a}{u^b}(u)$, is positive definite for each $u \in V$.
	\end{itemize}
The pair $(V,\F)$ is called a Minkowski normed space.
\end{definition}

The restriction of a Finsler function $F$ to a specific tangent space gives a Minkowski norm ${\F}_m: T_mM \to \R, v_m \mapsto F(m,u_m)$. Likewise, a smoothly varying family of Minkowski norms $\F_m: T_mM \to\R$ can be glued together to give a Finsler function $F: TM \to \R, (m,u_m) \mapsto \F_m(u_m)$.

The following proposition summarizes some of the most important properties of a Minkowski norm. 

\begin{proposition}
	Let $\F$ be a Minkowski norm on a vector space $V$. Then
	\begin{itemize}
		\item $\F(u)>0$ for all $u\neq0$,
		\item $\F(u_1+u_2) \leq \F(u_1)  + \F(u_2)$ , where equality holds if and only if $u_1=\alpha u_2$ with $\alpha \geq 0$,
		\item $x^i\fpd{\F}{u^i}(u)\leq \F(x)$ for all $x\neq 0$, where equality holds if and only if $x=\alpha u$ with $\alpha \geq 0$.
	\end{itemize}
\end{proposition}

These properties are usually referred to as \textit{positivity, triangle inequality} and \textit{the fundamental inequality}, respectively. An immediate consequence of the triangle inequality and the homogeneity is that any Minkowski norm $\F$ is a convex function. Moreover, the equality
\[
\F(t u + (1-t){\tilde u}) = t\F(u)+(1-t)\F({\tilde u})
\]
holds true if and only if $u=\alpha {\tilde u}$ with $\alpha \geq 0$ (see \cite{BCS2}).

In what follows, we will also need a converse statement, which is summarized in the following lemma.
 \begin{lemma}\label{strict}
	Let $\F$ be a regular non-negative real-valued function on a vector space $V$ that is moreover
	\begin{itemize}
	\item positive homogeneous: $\F(\lambda u)=\lambda \F(u)$ for all $u\in V$ and $\lambda >0$,
	\item convex: $\F(tu +(1-t){\tilde u}) \leq t\F(u)+(1-t)\F({\tilde u})$
	\item  strictly convex for all $u\neq\tilde{u}$, unless $u=\alpha\tilde{u}$ for a non-negative constant $\alpha$.
\end{itemize}
Then, $\F$ is a Minkowski norm. \end{lemma}
\begin{proof}
We only need to prove that $\F$ is  strongly convex everywhere (i.e.\ that $\E=\frac{1}{2}\F^2$ is strictly convex everywhere).
	First we show that $\F^2$ is convex. From the non-negativity and convexity of $\F$, it follows that for any $u,\tilde{u}\in V$ with $u\neq \tilde{u}$
	\[
	\left( \F(tu +(1-t){\tilde u})\right)^2 \leq \left(t\F(u)+(1-t)\F({\tilde u})\right)^2.
	\]
	After reordering the terms one gets the inequality
	\begin{equation}\label{reorderedconvex}
	\F^2(tu +(1-t){\tilde u}) \leq t\F^2(u)+(1-t)\F^2({\tilde u})-t(1-t)\left(\F(u)-\F(\tilde{u})\right)^2.
	\end{equation}

	Since the last term of the right-hand side is non-positive, the convexity of $\F^2$ follows. 
	
	To prove that $\F^2$ is strictly  convex everywhere, we need to show that equality only holds if $u=\tilde u$. In (\ref{reorderedconvex}) equality is satisfied only if the last term of its right-hand side is zero.  In that case, $\F(u)=\F(\tilde{u})$, and from the equality it also follows that 
	\[
	\F^2\left( (1-t)u+t\tilde{u}\right)=\F^2(u).
	\]
Therefore, $\F$ itself must be constant along the line $tu+(1-t)\tilde{u}$. But, $\F$ was assumed to be strictly convex unless $u=\alpha\tilde{u}$ with some non-negative $\alpha$. From $\F(u)=\F(\tilde{u})$ we see  (from the positive homogeneity of $\F$) that $\alpha=1$, and thus $u = \tilde u$. 
\end{proof}

A connected and open subset $D\subset V$ is said to be strictly convex if it contains the interior of any line segment joining any two points of the topological closure $\overline{D}$.

For a Minkowski vector space $(V,\F)$  we  define 
\[
B_\F= \big\{ u\in V ~ : ~ \F(u)<1 \big\},
\]
\[
S_\F=\big\{ u\in V ~ : ~ \F(u)=1 \big\},
\]
and we call $B_\F$, $\overline{B_\F}$ and $S_\F$ the open unit ball, closed unit ball and unit sphere  of the Minkowski normed space. In \cite{BCS2}, the authors mention that the open ball is a strictly convex domain of $V$. Furthermore, the closed ball is a compact set with boundary $S_\F$.

We recall from  \cite{Alv} the definition of a submersion between Minkowski spaces.

\begin{definition} \label{defMinsub}
	Let $(V_1,\F_1)$ and $(V_2,\F_2)$ be two Minkowski normed spaces.
	A surjective  linear map $\mu :V_1\rightarrow V_2$ is a submersion between Minkowski normed spaces if
	 \[\mu(\overline{B_{\F_1}})=\overline{B_{\F_2}}.
	\]
\end{definition}

We will also speak of a `Minkowski submersion', in short.
For later reference, we first prove a technical lemma.
\begin{lemma}\label{open}
	Let $V$ and $W$ be topological spaces, and $U$ an open subset of  $V$  with compact closure $\overline U$. Then, for any continuous map $\mu:V\rightarrow W$,
	\begin{equation}
		\mu (\overline{U})=\overline{\mu(U)}.
	\end{equation}
\end{lemma}

\begin{proof} Since $U\subset \overline U$, we have $\mu(U)\subset \mu(\overline U)$ and thus $\overline{\mu(U)}\subset \overline{\mu(\overline U)}$.
	Since any continuous image of a compact set is compact, $\mu (\overline{U})$ is compact, and therefore it equals its closure, so  $\mu (\overline{U})=\overline{\mu (\overline{U})}$. We conclude that $\overline{\mu(U)}\subset \mu (\overline{U})$.

	To show that also $\mu (\overline{U})\subset\overline{\mu(U)} $, we pick an element $\mu(x)$ of $\mu(\overline{U})$. Since $x\in\overline{U}$  there exists a sequence, say $x_n$ in $U$ that converges to $x$. Since $\mu$ is continuous, $x_n\rightarrow x$ implies $\mu(x_n)\rightarrow \mu(x)$. The sequence $\mu(x_n)$ is located in $\mu(U)$ by construction and therefore $\mu(x)$ is an accumulation point of $\overline{\mu(U)}$. Whence, $\mu(x)\in \overline{\mu(U)}$.
\end{proof}

The following proposition shows the relation between the two Minkowski norms, if a submersion  exists between them.

\begin{proposition}\label{basicprop}
	Let $\mu$ be a surjective linear map between the Minkowski normed spaces $(V_1,\F_1)$ and $(V_2,\F_2)$. Then the following statements are equivalent:
\begin{itemize}
\item[(i)] $\mu$ is a Minkowski submersion.
\item[(ii)] $\mu(\overline{B_{\F_1}})=\overline{B_{\F_2}}.$
\item[(iii)] $\F_2(v)=\inf \big\{ \F_1(u)~:~\mu(u)=v \big\}.$
\end{itemize}
	\end{proposition}

\begin{proof}   The equivalence between the statements (i) and (ii) is exactly Definition~\ref{defMinsub}. 

(ii)$\implies$(iii)~~  Let $u\in V_1$ be a nonzero vector. Since $\frac{u}{\F_1(u)} \in \overline{B_{\F_1}}$, it follows  from relation  (ii) that 
\[
	\F_2\left(\mu\left(\frac{u}{\F_1(u)}\right)\right) \leq 1.
\]
	 Using the linearity of $\mu$ and the positive homogeneity of $\F_2$, this leads to the inequality
\begin{equation}\label{ineq}
	\F_2(\mu(u))\leq \F_1(u)
\end{equation}
for all nonzero $u\in V_1$. Since  (\ref{ineq}) also holds true for $u=0$ trivially, it is valid for all $u\in V_1$.

Let now $v$ be a unit vector of $V_2$. Then, $v$ belongs to the closed unit ball of $V_2$ and therefore, by the property (ii), there exists an element  $u_0\in V_1$ with $\F_1(u_0)\leq1$ such that $\mu(u_0)=v$. As a result $u_0$ is an element of the set $\big\{ \mu^{-1}(v)\big\}$ and therefore
\begin{equation*}
\inf \big\{ \F_1(u)~:~\mu(u)=v  \big\} \leq 1.
\end{equation*}
On the other hand, (\ref{ineq}) implies that $\F_2(v)\leq \F_1(u)$, for all $u\in\mu^{-1}(v)$. Given that $v$ was a unit vector, we conclude that also
\begin{equation}\label{infimumisboundedfrombelow}
1 \leq \inf \big\{ \F_1(u)~:~\mu(u)=v \big\}.
\end{equation}
Notice that the norm of $u_0$ can not be strictly less than $1$, as it would contradict (\ref{infimumisboundedfrombelow}). As a result, the infimum of the set is included in the set, so it is actually a minimum, reached at $u_0\in V_1$, and the relation (iii) is therefore confirmed for unit vectors $v$. Multiplying the equation with an arbitrary positive constant and making use of the homogeneity of the Minkowski norms and the linearity of $\mu$, we conclude that (iii) is true for any element $v\in V_2$.

(iii)$\implies$(ii)~~ We  show first that $ \overline{B_{\F_2}}\subset\mu(\overline{B_{\F_1}}) $. Let $v \in B_{\F_2}$. Then, by assumption, 
\begin{equation*}
	\inf \big\{ \F_1(u)~:~\mu(u)=v \big\} ~=~ \F_2(v)<1.
\end{equation*}
It follows that there exists an element, say $u_0\in V_1$, for which $\F_1(u_0)<1$ and $\mu (u_0)=v$. Therefore, the open unit ball of $V_2$ is a subset of $\mu(B_{\F_1})$ and its closure, $\overline{B_{\F_2}}$, is  a subset of $\overline{\mu(B_{\F_1})}$. By Lemma~\ref{open} this set is the same as $\mu(\overline{B_{\F_1}})$.
 
Finally, we show that $\mu(\overline{B_{\F_1}}) \subset  \overline{B_{\F_2}}$. Since by the property (iii),  
\[
	\F_2(\mu(u))\leq \F_1(u), \qquad \forall u\in V_1,
\]
we see that if $v=\mu(u_0)$ for a  $u_0\in \overline{B_{\F_1}}$, also $\F_2(v)\leq 1$. This means that indeed $v\in \overline{B_{\F_2}}$.
\end{proof}

{\bf Example.} Let $\mu:V_1\rightarrow V_2$ be a surjective linear map between the Euclidean vector spaces $(V_1,\langle,\rangle_1)$ and $(V_2,\langle,\rangle_2)$.  The {\em horizontal subspace $\mathcal H$ of $\langle , \rangle_1$} is  the subspace of $V_1$ that is orthogonal to the kernel of $\mu$, that is: if $u\in {\mathcal H}$, then
$	\langle w, u\rangle_1=0$	for all $w \in {\rm Ker}\,\mu$. It is important to notice that a Euclidean vector space $(V_1,\langle,\rangle_1)$ is also a Minkowski normed space with the induced Minkowski norm $\F_1(u) :=\sqrt{\langle u,u\rangle_1}$.

The map $\mu$ is {\em a submersion between  Euclidean vector spaces} if
	\[
	\langle u,\tilde u \rangle_1 =\langle \mu(u),\mu(\tilde u)\rangle_2,
	\]
 for all $u,\tilde u\in {\mathcal H}$. 
In other words, the definition says that $\mu$ is an isometry between the horizontal subspace ${\mathcal H}$ of $V_1$ and the vector space $V_2$.

Proposition~\ref{basicprop} can be used to see that a Euclidean submersion $\mu: (V_1,\langle,\rangle_1) \to (V_2,\langle,\rangle_2)$ is a special case of a Minkowskian submersion w.r.t.\ the two induced norms, $\mu: (V_1,\F_1) \to (V_2,\F_2)$.  Denote by $u^{hor}$ and $u^{ver}$ the horizontal and vertical part of a $u\in V_1$. Since $\langle u^{hor},u^{ver} \rangle_1=0$, the square of the norm of $u$ equals
\[
(\F_1(u))^2 = (\F_1(u^{hor}))^2 + (\F_1(u^{ver}))^2.
\]
On the other hand, when we write the definition of a Euclidean submersion in terms of the induced norm, we see (after taking a square root) that $\F_1(u^{hor}) =\F_2(\mu(u^{hor}))= \F_2(\mu(u))$. We conclude that $\F_1(u) \geq  \F_2(\mu(u))$. As a consequence, for any $v\in V_2$, 
\[
\F_2(v)=\min \{\F_1(\mu^{-1}(v))\}.
\]
Proposition~\ref{basicprop} now shows that $\mu$ is a submersion of Minkowski spaces.

\section{Nonlinear lifts} \label{nonlinearliftssection}

Let $\mu: V_1 \to V_2$ be a surjective linear map between two vector spaces $V_1$ and $V_2$. The next definition can be seen as the special case of a nonlinear splitting, when the manifolds are in fact vector spaces.
\begin{definition} \label{defnonlinearlift}
A nonlinear lift is a continuous map $\h: V_2\to V_1$ which is smooth on $V_2\setminus \{0\}$ and which is such that $\mu\circ \h = {\rm id}$.

A nonlinear lift is homogeneous when $\h(\lambda v) = \lambda \h(v)$, for all $\lambda>0$ and $v\in V_2$. The set ${\rm Im}\,\h$ is then called the horizontal cone of the lift.
\end{definition}

Since $\h$ has a left inverse, it is injective. This definition is reminiscent to the one of a nonlinear splitting in Section~\ref{secnonlinsplit}. It is clear that when $h: \pi^*TN \to TM$ is a nonlinear splitting, then, for each $m\in M$ (with $\pi(m)=n)$, the map ${\h}_m:T_{n}N \to T_{m}M, v_n \mapsto h(m,v_n)$ is a nonlinear lift. On the other hand, a smoothly varying family of nonlinear lifts  ${\h}_m:T_{\pi(m)}N \to T_{m}M$ constitute together the nonlinear splitting $h: \pi^*TN \to TM, (m,v_n) \mapsto \h_m(v_n)$.

We now come back to the Minkowski norms of Proposition~\ref{basicprop}.
In practice, although the surjective linear map $\mu$ is often given beforehand,  the norm $\F_2$ on $V_2$ is not. The next proposition shows how to construct  a Minkowski norm on $V_2$ from $\F_1$.

\begin{proposition} \label{subduced}
	Let $(V_1,\F_1)$ be a Minkowski space, $V_2$ a vector   space, and $\mu : V_1\rightarrow V_2$ a linear and surjective map. Then,
	
	\begin{enumerate} \item[(i)] there exists an induced homogeneous nonlinear lift $\h: V_2 \to V_1$,
	
\item[(ii)]	there  exists a unique Minkowski norm $\F_2:V_2\rightarrow\R$ on $V_2$, given by
\begin{equation}\label{def}
	\F_2(v):=\inf \big\{ \F_1(u)~:~\mu(u)=v \big\},
\end{equation}
referred to as {\em the subduced norm}, for which   $\mu$ is a submersion between the Minkowski spaces,
\item[(iii)] the nonlinear lift $\h$ and the Minkowski norms $\F_1$ and $\F_2$ are related by 
\begin{equation} \label{eleven}
\F_1(\h(v))=\F_2(v).
\end{equation}
\end{enumerate}
\end{proposition}

\begin{proof}
Let us first define $\F_2:V_2\rightarrow\R$ by means of the relation (\ref{def}).
Since the value of $\F_2$ is defined as the infimum of a set which consists only of non-negative numbers, $\F_2$ itself must be non-negative. It is also clear that $\F_2$ is positive homogeneous of degree 1, since 
\begin{eqnarray*}
\F_2(\lambda v)&=&\inf \big\{ \F_1(u)~:~\mu(u)=\lambda v \big\} = \inf \big\{ \F_1(\lambda \tilde u)~:~\mu(\lambda \tilde u)=\lambda v \big\}\\ &=&\inf \big\{ \lambda \F_1( \tilde u)~:~\mu( \tilde u)= v \big\}= \lambda \F_2(v)\qquad\mbox{for all}\quad\lambda>0.
\end{eqnarray*}
We will show in the second part of the proof that $\F_2$ is also strongly convex and regular, and is therefore a Minkowski norm. Proposition~\ref{basicprop} then  guarantees that $\mu$ is a submersion between the Minkowski normed spaces, and it also guarantees that the norm $\F_2$ on $V_2$ is unique. We start the proof, however, with the construction of an associated homogeneous nonlinear lift.

(i) Let $v$ be an arbitrary nonzero element of $V_2$. Since $\mu^{-1}(v)$ does not contain the zero vector of $V_1$, it follows that $\F_2(v)$ is positive.(Indeed, if $\F_2(v)$ was zero, the zero vector of $V_1$ would be an element or an accumulation point of the set $\mu^{-1}(v)$. But this set is closed and therefore includes its accumulation points. We arrive at a contradiction with the fact that $\mu^{-1}(v)$ does not contain the zero vector of $V_1$, therefore it follows that $\F_2(v)$ is indeed positive.) By homogeneity, we conclude that the set $\F_2^{-1}(1)$ is non-empty. Let us denote the convex hull of it by $\overline{B_{\F_2}}$ and its interior by $B_{\F_2}$. In a similar fashion as in the proof of Proposition~\ref{basicprop}, we can conclude from the definition of $\F_2$ that $B_{\F_2}\subset\mu(B_{\F_1})$ and  from the inequality
$
		\F_2(\mu(u))\leq \F_1(u),
$
that  $\mu(B_{\F_1})\subset B_{\F_2}$. For these reasons, $
	\mu(B_{\F_1})=B_{\F_2}
$, and therefore also $
	\mu(\overline{B_{\F_1}})=\overline{B_{\F_2}}
$. As a result, $\mu$ is surjective from $\overline{B_{\F_1}}$ to $\overline{B_{\F_2}}$ and therefore, for any $v\in V_2$ with $\F_2(v)=1$, there exists a vector $u_0\in V_1$ with $\F_1(u_0)=1$ and $\mu (u_0)=v$. When restricted to a fibre $\mu^{-1}(v)$, the function $\F_1$ is strictly convex. Since at $u_0$ this function reaches its minimum, the minimum $u_0$ must be unique.

We conclude that the infimum in the definition of $\F_2$ is a minimum. In what follows we define the nonlinear lift $\h: V_2 \to V_1$ as the map $v \mapsto \h(v)=u_0$, that is: $\h(v)$ is the unique vector in $\mu^{-1}(v)$ for which the equality
\[
\F_1(\h(v))=\F_2(v)
\]
is valid. (This then corresponds to the expression (\ref{eleven}) in statement (iii).)

It is clear from the above equation that, for all $\lambda >0$, 
\[
\F_1(\h(\lambda v))=\F_2(\lambda v) = \lambda \F_2(v)=\lambda \F_1(\h(v)).
\]
By the uniqueness of the nonlinear lift we may conclude that the map $h$ is positive homogeneous of degree 1.

We show next that this nonlinear lift is smooth outside of the origin. For this reason, fix a basis $\{E_i,E_\alpha\}$ of $V_1$ such that the vectors $\{E_\alpha\}$ span $\ker \mu$ and such that the vectors $\{\mu(E_i)\}$ form a basis of $V_2$. We will use the notation $\{u^a\}=\{v^i,w^\alpha\}$ for the coordinates on $V_1$, with respect to this basis. Since the restriction of $\F_1$ to $\mu^{-1}(v)$ reaches its unique minimum at $\h(v)$, we know that
\begin{equation}\label{minn}
\frac{\partial \F_1}{\partial w^{\alpha}}(\h(v))=0.
\end{equation}
From this, also \begin{equation} \label{minn2}
\frac{\partial \E_1}{\partial w^{\alpha}}(\h(v))=0.
\end{equation}  Since  the Hessian of $\E_1=\frac{1}{2}\F_1^2$ is positive definite, so is the submatrix \[
\left(\frac{\partial^2 \E_1}{\partial w^{\alpha}\partial w^{\beta}}\right).
\] By means of the implicit function   theorem we know that there will be a unique smooth map $\h$ (in an open neighborhood of $w$) that  satisfies relation (\ref{minn2}) and therefore also (\ref{minn}). Since the nonlinear lift is unique, it must coincide with this smooth map.

(ii) Since the nonlinear lift $\h$ and the Minkowski norm $\F_1$ are both smooth outside of the origin, we conclude that $\F_2$ is regular. In order to show that $\F_2$ is a Minkowski norm, we will rely on Lemma~\ref{strict}. 
Let $v$ and $\tilde v$ be elements of $V_2$ that are not positive scalar multiples of each other, and let $t\in ]0,1[$. Then, 
\begin{eqnarray*}\F_2\Big(tv+(1-t){\tilde v}\Big)&=&\F_2\Big(\mu \big(\h(t v)\big)+\mu\big(\h((1-t){\tilde v})\big)\Big)= \F_2\Big(\mu \big(\h(t v)+\h((1-t){\tilde v})\big)\Big)\\ &\leq& \F_1\Big(\h(t v)+\h((1-t){\tilde v})\Big) \\ &<& \F_1\Big(\h(t v)\Big)+\F_1\Big(\h((1-t) {\tilde v})\Big)\\&=&t\F_1\Big(\h(v)\Big)+(1-t)\F_1\Big(\h({\tilde v})\Big)=t\F_2(v)+(1-t)\F_2 ({\tilde v}).
\end{eqnarray*}
In the first line we have used $\mu\circ \h = {\rm id}$ and the linearity of $\mu$. The second line relies on the inequality $\F_2(\mu(u))\leq \F_1(u)$. In the third line we have used the triangle inequality. It is a strict inequality because $\h(t v)$ and $\h((1-t) {\tilde v})$ are not positive scalar multiples of each other. If that were the case, $\h(t v) = \beta \h((1-t) {\tilde v})$, then from the injectivity and the homogeneity of $\h$,  $t v  = \beta  (1-t) {\tilde v}$, that is: $v$ and $\tilde v$ are  positive scalar multiples of each other. In the last line we have used the homogeneity of $\F_1$ and $\h$, and finally the defining relation (\ref{eleven}) of $\h$.
The conclusion is that $\F_2$ satisfies the conditions of Lemma~\ref{strict} and that it is a Minkowski norm. \end{proof}

{\bf Example 1.} Consider the following Randers-type Minkowski norm on $\R^2$
\begin{equation*}
	\F_1(v,w)=\sqrt{2v^2+2vw+2w^2}+v+w, \qquad (v,w)\in \R^2
\end{equation*}
and let $\mu$ be the (surjective and linear) projection to the $v$-axis, $\mu(v,w)=v$. Then, we know from Proposition~\ref{subduced} that $\mu$ and $\F_1$ determine a nonlinear lift $\h: \R \to \R^2$ and a subduced Minkowski norm $\F_2$ on $\R$. We calculate them explicitly.

The circle, centered at the point $(-1,-1)$ with radius $r=\sqrt{3}$ constitutes the set of unit-length vectors for $\F_1$.  We may find an explicit solution for $\h$ by solving
\[
\fpd{\F_1}{w}(\h(v)) = 0.
\]
An explicit calculation easily reveals that the couple $\h(v) \in \R^2$ should lie in the half plane $v+2w\leq0$  and that 
\[
\h(v)= \begin{cases} 
(v,\frac{1}{1+\sqrt{3}}v), &\, v\leq 0, \\
(v,\frac{1}{1-\sqrt{3}}v), &\, v>0.
\end{cases}
\]
The horizontal cone ${\rm Im}\,\h$ consists therefore of the two horizontal unit vectors $u_1=(-1+\sqrt{3},-1)$ and $u_2=(-1-\sqrt{3},-1)$ of $\F_1$, and of their nonnegative scalar multiples. The figure below shows the half plane $v+2w\leq 0$ (below the dotted line $v+2w=0$), the horizontal unit vectors $u_1$ and $u_2$ and the unit sphere of $\F_1$. 
\begin{center}
	\begin{tabular}{c}
		\begin{minipage}{11cm}
			\begin{center}
				\includegraphics[height=5.8cm]{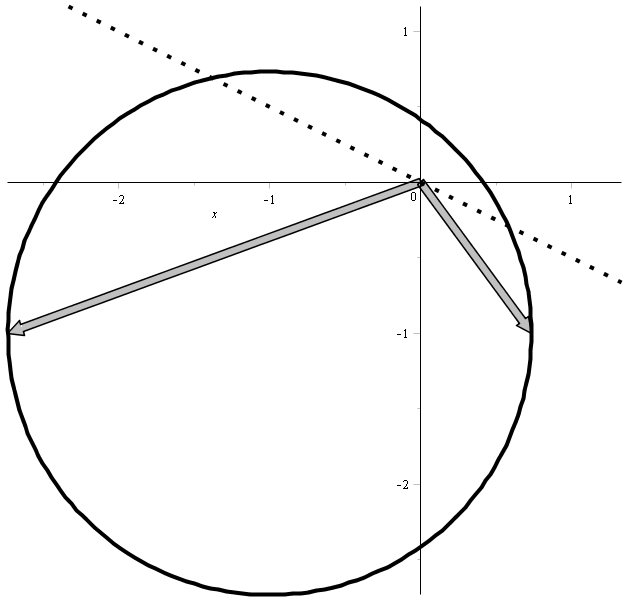}
			\end{center}			
		\end{minipage}	
	\end{tabular}
\end{center}
The subduced metric  $\F_2=\min_{u\in \mu^{-1}(v)}\{\F_1(u)\}
$ on $\R$ can be explicitly given by
\[\F_2(v)= \begin{cases} 
\frac{1}{-1-\sqrt{3}}v, &\, v\leq 0, \\
\frac{1}{-1+\sqrt{3}}v, &\, v>0.
\end{cases}
\]
It
admits two unit-length vectors, $\mu_1(u_1)=-1-\sqrt{3}$ and  $\mu_1(u_2)=-1+\sqrt{3}$. Their convex hull is therefore the unit sphere of $(\R,\F_2)$.

{\bf Example 2.} Consider again the case of a Euclidean submersion $\mu:(V_1,\langle,\rangle_1) \to (V_2,\langle,\rangle_2)$. In the previous section we have seen that we may interpret it also as a Minkowski submersion, with $\F_1(u):=\sqrt{\langle u,u\rangle_1}$ and $\F_2(v):=\sqrt{\langle v,v\rangle_2}$. As such, it defines a corresponding nonlinear lift $\h: V_2\to V_1$.  From the proof of Proposition~\ref{subduced} we know that  $\h(v)\in \pi^{-1}(v)$ satisfies $\F_1(\h(v))=\F_2(v)$. When expressed in terms of the Euclidean norms, we get  
\[
\langle \h(v), \h(v) \rangle_1= \langle v, v \rangle_2. 
\]
If we set  $\h(v) = (\h(v))^{hor}+ (\h(v))^{ver}$, the previous relation leads to $\langle (\h(v))^{ver}, (\h(v))^{ver} \rangle_1=0$, and thus $\h(v)^{ver}=0$, or $\h(v)\in {\mathcal H}_1$. 
The nonlinear lift of $v\in V_2$ is therefore, as expected, the unique horizontal element $u\in {\mathcal H}$ with $\mu(u)=v$. It so happens here that the nonlinear lift is a linear lift.

We may also conclude that the Minkowski norm $\F_2$ on $V_2$ that is induced by the inner product $\langle,\rangle_2$ is the same as the subduced Minkowski norm of $\F_1$ that is guaranteed by Proposition~\ref{subduced}.

\section{Associated Euclidean submersions}

In the previous section we have already mentioned that a Euclidean norm  always provides a Minkowski norm on the same vector space. On the other hand, when a Minkowski normed space $(V,\F)$ is given, we may associate a Euclidean vector space $(V,g_u)$ to each $u\in V$, after taking the identifications $T_uV_1\simeq V_1$ and $T_u\mu\simeq \mu$ into account. If $\{E_a\}$ is a basis for $V$, and if we denote the corresponding coordinates by $u^a$, then
\[
g_u (x,y) := \spd{\E}{u^a}{u^b}(u)x^ay^b, \qquad \forall x,y \in T_uV \simeq V.
\]

We now come back to the case where a submersion $\mu:(V_1,\F_1)\rightarrow (V_2,\F_2)$ of Minkowski spaces is given.

\begin{proposition} \label{Euclprop}
	Let $\mu:(V_1,\F_1)\rightarrow (V_2,\F_2)$ be a submersion of Minkowski spaces with corresponding nonlinear lift $\h: V_2 \to V_1$. 
	
	\begin{enumerate} \item[(i)] For each $v\in V_2$, the tangent space $T_{\h(v)}({\rm Im}\,\h)$ to the horizontal cone, is the horizontal subspace of the Euclidean norm $g^1_{\h(v)}$ in $V_1$.

\item[(ii)]	For each $v\in V_2$,  $\mu$ is a submersion between the Euclidean spaces $(V_1,g^1_{\h(v)})$ and $(V_2,g^2_{v})$.

\item[(iii)] The linear map $T_{\h(v)}\h$ is the  nonlinear lift $V_2 \to V_1$ that corresponds to the associated  Minkowski submersion of (ii). 
	\end{enumerate}

\end{proposition}

\begin{proof} In this proof we will use  a basis $\{E_i,E_\alpha\}$ for $V_1$ such that $\{E_\alpha\}$ spans $\ker \mu$ and $e_i=\mu(E_i)$ forms a basis of $V_2$. We will use $\{v^i,w^\alpha\}$ and $\{v^i\}$ as coordinates on $V_1$ and $V_2$. A nonlinear lift then takes the form 
\[
\h(v) = (v^i,\h^\alpha(v)).
\]
Recall that, in the current case of interest, $\h$ satisfies expression (\ref{minn2}).

 (i) We need to show that the subspace $T_{\h(v)}({\rm Im}\,\h)$ is the orthogonal complement to $\ker(T_{\h(v)}\mu) \simeq \ker\mu$ with respect to $g^1_{\h(v)}$. We show first that 
\begin{equation}\label{orthogonal}
	g^1_{\h(v)}(T_v\h(x),E_{\alpha})=0, \qquad \forall z\in V_2.
\end{equation}
For $x=x^ie_i$, we have  $T_v\h(x)=x^i\left(E_i+\frac{\partial \h^{\alpha}}{\partial v^i}(v)E_\alpha\right)$.
If we set \[g^1_{i\alpha}=g^1(E_i,E_j) = \spd{\E_1}{v^i}{v^j}, \quad
g^1_{i\alpha}=g^1(E_i,E_\alpha) = \spd{\E_1}{v^i}{w^\alpha},\quad g^1_{\alpha\beta}=g^1(E_{\alpha},E_\beta) = \spd{\E_1}{w^\alpha}{w^\beta}
,
\] 
we easily find that, in view of relation (\ref{minn2}), 
\[
g^1_{\h(v)}(T_v\h(x),E_{\alpha}) = x^i\left(g^1_{i\alpha}(\h(v)) +g^1_{\alpha\beta}(\h(w))\frac{\partial \h^{\beta}}{\partial v^i}(v) \right) =  x^i\frac{\partial}{\partial v^i}\left(\frac{\partial \E_1}{\partial w^{\alpha}}(\h(v))\right)=  0.
\]
This shows that $T_{\h(v)}({\rm Im}\,\h)$ is a subset of the horizontal space. It is in fact the whole space. Indeed, for any other horizontal vector $u \in T_{\h(v)} V_1 \simeq V_1$ with $T_{\h(v)}\mu(u) =x\in T_vV_2\simeq V_2$, the difference  $u-T_v\h(x)$ is $T_{\h(v)}\mu$-vertical. Since both are supposed to be horizontal we get that $g^1_{\h(v)}(u-T_v\h(x),E_{\alpha})=0$, which can only happen if $u=T_v\h(x)$.

(ii) We need to show, that for any $v,x,y\in V_2$ 
	\begin{equation}\label{relation}
		g^2_v(x,y)=g^1_{\h(v)}(T_v\h(x),T_v\h(y)).
	\end{equation}

 Since $\F_2 =\F_1\circ \h$, we also have $\E_2 =\E_1\circ \h$. From this
\[
\fpd{\E_2}{v^i} = \fpd{\E_1}{v^i}\circ \h + \left(\fpd{\E_1}{w^\alpha}\circ \h\right) \fpd{\h^\alpha}{v^i}
 = \fpd{\E_1}{v^i}\circ \h ,
\]
in view of expression (\ref{minn2}). A further derivative gives
\[
g^2_v(x,y)=x^iy^j\left (\spd{\E_2}{v^i}{v^j} \right) = x^iy^j\left( \spd{\E_1}{v^i}{v^j}\circ \h + \left(\spd{\E_1}{v^i}{w^\beta}\circ \h\right) \fpd{\h^\beta}{v^j}    \right). 
\]
On the other hand, if we take into account that in \[
T_v\h(y)=y^j\left(E_j+\frac{\partial \h^{\beta}}{\partial v^j}(v)E_\beta\right)
\] the second term is vertical, we also get
\[
g^1_{\h(w)}(T_v\h(x),T_v\h(y)) = g^1_{\h(w)}(T_v\h(x),y^jE_j) = x^iy^j\left( \spd{\E_1}{v^i}{v^j}\circ \h + \left(\spd{\E_1}{v^i}{w^\beta}\circ \h\right) \fpd{\h^\beta}{v^j}    \right)
\]
for the right-hand side.

(iii) The example at the end of Section~\ref{minkowskinormssection} shows that a Euclidean submersion is a special case of a Minkowski submersion. As such, we saw that Proposition~\ref{subduced} guarantees that there exists a corresponding nonlinear lift. Let us denote, for the Euclidean submersion of (ii), this nonlinear lift by ${\tilde \h}_v: V_2 \to V_1$. In the second example of Section~\ref{nonlinearliftssection} we have seen that seen that this lift is in fact linear, and that (in the current situation) it can be determined from  
\[
	g^1_{\h(v)}({\tilde\h}_v(x),{\tilde\h}_v(y)) = g^2_v(x,y).
\]
From expression (\ref{relation}), it is clear that ${\tilde \h}_v = T_{\h(v)}\h$.
 \end{proof}

The first property of Proposition~\ref{Euclprop} had already appeared in the literature, e.g.\ it features in some form in \cite{libing}. It is clear from the discussion above, that the orthogonal complement with respect to $g^1_{\h(v)}$ to the kernel of $\mu$ is not the horizontal cone, but rather the tangent space of the cone at  $\h(v)$.

\section{Finsler submersions} \label{Sec5}

In this section we come back to the situation of a fibre bundle $\pi: M \to N$ and a Finsler function $F$ on $M$. We have already mentioned that the restriction of a Finsler function to a specific tangent space, $F|_{T_mM}$ is a Minkowski norm $\F_m$ on the vector space $T_mM$.

The concept of a submersion between Finsler manifolds has been introduced in \cite{Alv} as a generalization of Riemannian submersions. It can be thought of as the analogue of submersions of Minkowskian spaces, to the context of Finsler manifolds. 
\begin{definition}\label{finslersub}
	A surjective map $\pi:(M,F_1)\rightarrow (N,F_2)$ between two Finsler manifolds is a Finsler submersion, if the tangent map $\mu=T_m\pi:T_mM\rightarrow T_{\pi (m)}N$ is a submersion between the Minkowski spaces $(T_mM, (\F_1)_m)$ and $(T_{\pi(m)}N, (\F_2)_{\pi(m)})$, for each $m\in M$. If such a submersion exists, we call $F_2$ the subduced Finsler function of $F_1$.
\end{definition}

It is important to realize that in this definition the two Finsler functions $F_1$ and $F_2$ are supposed to be given from the outset.
In Proposition~\ref{subduced} we have shown that a surjective linear map between a Minkowski normed space $(V_1,\F_1)$ and the vector space $V_2$ uniquely determines a subduced norm $\F_2$, making $(V_2,\F_2)$ a Minkowski normed space. However, in general, a subduced Finsler function can not be constructed in the same way. In fact, if $m ,\tilde m \in\pi^{-1}(n)$ are two points of $M$ located in the same fibre, both $T_{m}\pi$ and $T_{\tilde m}\pi$ give rise to (in general) different subduced Minkowski norms  on $T_{n}N$ (which we will denote by $(\F_2)_{m}$ and $(\F_2)_{\tilde m}$). One should for this reason be aware that in Definition~\ref{finslersub}, it is inherently assumed that all these subduced Minkowski norms coincide, and that they are the same as the restriction  to $T_nN$ of the given Finsler function $F_2$ (which we will denote by $(\F_2)_n$). 

In the next proposition we give a sufficient condition for the existence of a subduced Finsler function $F_2$. Recall that in Section~\ref{secnonlinsplit} we have associated  to a Finsler function $F_1$ a homogeneous nonlinear splitting $h:\pi^*TN \to TM$ by means of expression (\ref{definingrelation}). If we consider the restriction of this $h$ to a fixed $m=(x^i,y^\alpha)\in M$,  $\h_m: T_{\pi(m)}N \to T_mM, (v^i) \mapsto (v,h^\alpha(x,y,v))$, then $\h_m$ defines a nonlinear lift in the sense of Definition~\ref{defnonlinearlift}. In fact, if we define likewise $(\E_1)_{m}: T_m M\rightarrow\R \mbox{\,\,by\,\,} (v^i,w^\alpha) \mapsto E_1(x,y,v,w)$, then we can write expression (\ref{definingrelation}) also as 
\[
\fpd{(\E_1)_{m}}{w^\alpha}\circ \h_m  =0.
\]
When compared to relation (\ref{minn2}), this shows that each $\h_m$ is, in fact, the nonlinear lift that one can associate to the Minkowski norm $(\F_1)_m$.

\begin{proposition}\label{finslersplitting}
	Let $h:M\rightarrow N$ be the homogeneous splitting induced by the Finsler function $F_1$ on $M$ and assume that
	the canonical spray $S_1$ is tangent to the horizontal submanifold ${\rm Im}\,h$ (i.e.\ $
	\clift{Y}(E_1)\circ h=0$). Then the function
	\[
	F_2:=F_1\circ h
	\]
	determines a   function on $TN$. Moreover, it is a Finsler function with the property that $\pi$ is a submersion between the Finsler manifolds $(M,F_1)$ and $(N, F_2)$.
\end{proposition}

\begin{proof}
In view of the the condition $	\clift{Y}(E_1)\circ h=0$,  and in view of the definition of $h$ in (\ref{definingrelation}), we see that 
\[
\fpd{(E_1\circ h)}{y^\alpha} = \fpd{E_1}{y^\alpha} \circ h + \left(\fpd{E_1}{w^\beta}\circ h\right) \fpd{h^\beta}{y^\alpha}=0.  
\]
The functions $E_1\circ h$ and $\bar F_2:=\sqrt{2E_1\circ h}$ can for this reason be interpreted as  functions on $TN$. Since it is a composition of smooth functions (except on the zero section), it is regular. $\bar F_2$ is a composition of  homogeneous functions, and therefore it is homogeneous itself. 
Moreover, 
\begin{eqnarray*}
0 &=& \fpd{}{y^\beta} \left( \frac{\partial E_1}{\partial w^{\alpha}} \circ h \right) = \left(\spd{E_1}{y^\beta}{w^\alpha}\right)\circ h+ \left(\spd{E_1}{w^\alpha}{w^\gamma} \circ h\right) \fpd{h^\gamma}{y^\beta} \\&=& \fpd{}{w^\alpha}\left( \fpd{E_1}{y^\beta} \circ h \right) +  \left(\spd{E_1}{w^\alpha}{w^\gamma} \circ h\right) \fpd{h^\gamma}{y^\beta} = \left(\spd{E_1}{w^\alpha}{w^\gamma} \circ h \right)\fpd{h^\gamma}{y^\beta}.
\end{eqnarray*}
Since the matrix $\left(\spd{E_1}{w^\beta}{w^\gamma}\right)$ is always non-degenerate, we may conclude from this that  $\partial{h^\gamma}/\partial{y^\beta} =0$. For this reason, the induced nonlinear lifts that one can associate to either $m$ or $\tilde m$ (with  $\pi(m)=\pi(\tilde m)$) are equal: $\h_m=\h_{\tilde m}$.

We are therefore in the situation of Proposition~\ref{subduced} and in each tangent space $T_nN$ we can construct (the same) Minkowski norm $(\F_2)_m = (\F_2)_{\tilde m}$ (which we may now denote as $(\F_2)_n$, for convenience) from either one of the submersions $T_m\pi$ between the Minkowski spaces $(T_m M,(\F_1)_m)$ and $(T_n N,(\F_2)_n)$ in the fibre over $n$ (i.e.\ for each choice of $m\in\pi^{-1}(n)$). The collection over all $n$ of all these norms $(\F_2)_n$ gives by construction the function $F_2=F_1\circ h$, we had defined above.  For this reason $F_2$ is a Finsler function. 
\end{proof} 

\textbf{Example.} Assume that a connected Lie group $G$ acts freely and properly on a Finsler manifold $(M,F)$ and that the Finsler function $F$ is invariant under the induced action of $G$ on $TM$. Then, $\pi: M\rightarrow M/G=N$ is a principal fibre bundle. Keeping the notation of the previous paragraphs (following e.g.\ \cite{Inv}) the invariance property of $F$ can be locally expressed as
\[
\clift{{\xi}_M}(E)=0, \qquad \forall \xi\in\la.
\]
Here, $\xi_M$ stands for the fundamental vector field on $M$ of $\xi$,  associated to the group action. Since one may form a basis for the set of $\pi$-vertical vector fields on $M$ consisting of only such vector fields, we have that  $\clift{Y}(E) =0$, i.e.\ condition  (\ref{symmetrycondition}) is satisfied on the whole of $TM$. Proposition~\ref{finslersplitting} then guarantees the existence of a subduced Finsler function on $N$.

In the next sections, we indicate how the results from the previous sections can be used to provide new examples of Finsler functions on homogeneous spaces.

\section{Finsler submersions to reductive homogeneous spaces} \label{homspacesec}

Let $M=G$ be a connected Lie group (with unit element $e$),  $H<G$ be a closed subgroup of $G$ and $G/H$ the corresponding space of left cosets $gH$ (with notation $o=eH$). We will denote the natural projection  by $\pi:G\to G/H, g \to gH$. Then,  $G/H$ inherits a transitive left action of $G$, given by $
\tau_{g_1}(g_2H)=(g_1g_2)H
$. It is well-known (see e.g.\ \cite{Lee}) that there exists a unique differentiable structure on $G/H$ such that $\pi$ is a smooth submersion. We will denote this  homogeneous space by $N=G/H$. 

We will denote the Lie algebras of $G$ and $H$ by $\mathfrak{g}$ and $\mathfrak{h}$, respectively. The kernel of $T_e\pi: \mathfrak g \to T_oN$ is the subspace $\mathfrak{h}$.
Straightforward computation shows that
\begin{eqnarray}\label{Ad}
	T_e\pi\circ \text{Ad}_h &=& T_o\tau_h \circ T_e\pi, \qquad \forall h\in H, \\
T_o\tau_g\circ T_e\pi &=&  T_g\pi \circ T_eL_g 	\qquad \forall g\in G\mbox{\,\,\,with \,\,\,} \pi(g)=o.  \label{general}
\end{eqnarray}

In the remainder of the paper, we will assume that  the homogeneous space $G/H$ is {\em reductive}, that is: that there exists a subspace $\mathfrak{m}$ of $\mathfrak{g}$, such that
$
\mathfrak{g}=\mathfrak{h}\oplus\mathfrak{m},
$
and such that the subspace $\mathfrak{m}$ is invariant under the adjoint action of $H$ on $\mathfrak{g}$. Then, $[\mathfrak{h},\mathfrak{m}]\subset \mathfrak{m}$.  

The restriction of the map $T_e\pi$ to $\mathfrak m$ defines an isomorphism between $\mathfrak{m}$ and $T_o(G/H)=T_oN$. When we want to make explicit use of this isomorphism, we will write $\beta :\mathfrak{m}\rightarrow T_oN,  \zeta \mapsto T_e\pi (\zeta)$ and $\beta^{-1}$ for its inverse. Since $\text{Ad}_h\zeta\in \mathfrak m$ when $\zeta\in \mathfrak m$, relation (\ref{Ad}) becomes
$
\beta(\text{Ad}_h \zeta)= T_o\tau_h (\beta(\zeta)), 
$ for all $\zeta \in \mathfrak m$, or
\begin{equation}\label{Ad2}
 \text{Ad}_h \circ \beta^{-1} = \beta^{-1}\circ T_o\tau_h, \qquad \forall h\in H.
\end{equation}

	A Finsler function $F: TG \to\R$ on a Lie group $G$ is said to be {\em left-invariant} if left translations are isometries:
	\[
		F ((e,\xi))=F((g,T_eL_g(\xi))), \qquad \forall (g,\xi) \in G\times \mathfrak g.
	\]
 $F$ is {\em bi-invariant} if both left and right translations are  isometries.

It is easy to see that there exists a one-to-one correspondence between  left $G$-invariant Finsler metrics $F$ on  $G$ and Minkowski norms $\F_e$ on the Lie algebra $\mathfrak{g}$. The relation is given by
\[
\F_e(\xi) = F(e,\xi)\qquad \mbox{and} \qquad F(g,v_g) = \F_e(T_gL_{g^{-1}}v_g).
\]
 In case $F$ is assumed to be left $G$-invariant, but only right $H$-invariant, the corresponding Minkowski norm $\F_e$ on $\la$ is $\text{Ad}(H)$-invariant (and vice versa). Indeed, 
\[
\F_e(\text{Ad}_h\xi) = F(e, \text{Ad}_h\xi) = F(h^{-1}eh, T_hL_{h^{-1}} T_eR_{h}(\text{Ad}_h\xi) ) = F(e,\xi) = \F_e(\xi).
\]
A small extension of this property shows that bi-invariant Finsler function $F$ are in one-to-one correspondence with  $\text{Ad}(G)$-invariant Minkowski norms on $\la$.

In the next proposition we make use of our framework to guarantee the existence of a subduced Finsler function.
\begin{proposition}\label{mainn}
Assume that $G/H$ is reductive  and that $F_1$ is a Finsler function on the Lie group $G$, that is left $G$- and right $H$-invariant. Then, 
\begin{itemize} \item[(i)] there exists an $\text{Ad}(H)$-equivariant homogeneous nonlinear lift $\bar \h: \mathfrak m \to \mathfrak g$ for the linear surjective map $p_{\mathfrak m}: {\mathfrak g} \to {\mathfrak m}$,

\item[(ii)]
there exists a
 unique  subduced, $G$-invariant Finsler metric $F_2$ on the homogeneous space $G/H$, such that $\pi: (G,F_1) \to (G/H,F_2)$ is a Finsler submersion,
	\item[(iii)] the homogeneous nonlinear  splitting $h: \pi^*T(G/H) \to TG$ of the submersion is related to the nonlinear lift $\bar\h$ of (i) by means of
	\[
	h(g,v_n) = T_eL_g {\bar \h} \left(\beta^{-1}\left(T_n\tau_{g^{-1}}(v_n)\right)\right).
	\] 
\end{itemize}	
\end{proposition}

\begin{proof}
(i)	Under the assumptions of the proposition, $F_1$  determines an $\text{Ad}(H)$-invariant Minkowski norm $(\F_1)_e$ on $\mathfrak{g}$. Consider the surjective linear map $p_{\mathfrak{m}}: \mathfrak g \to \mathfrak m$, with $p_{\mathfrak{m}}=\beta^{-1}\circ T_e\pi$.  By Proposition~\ref{subduced}, this Minkowski norm determines a unique subduced Minkowski norm $\bar{\F}_2$ on $\mathfrak{m}$ by
	\begin{equation}\label{F2}
		{\bar \F}_2(\zeta)=\inf \{(\F_1)_e(\xi)~|~p_{\mathfrak{m}}(\xi)=\zeta \}.
	\end{equation}
First, we show that this Minkowski norm is invariant under the induced adjoint action of $H$.
	 	Since $\text{Ad}_h:\mathfrak{g}\rightarrow\mathfrak{g}$ is bijective, we have
	\begin{eqnarray}
	\nonumber {\bar\F_2}(\text{Ad}_h(\zeta))&=&\inf \{(\F_1)_e(\xi)~|~p_{\mathfrak{m}}(\xi)=\text{Ad}_h(\zeta) \} \\
	\nonumber &=&\inf \{(\F_1)_e(\text{Ad}_h(\xi))~|~p_{\mathfrak{m}}(\text{Ad}_h(\xi))=\text{Ad}_h(\zeta) \}.
	\end{eqnarray}
	In view of expression (\ref{Ad2}) $p_{\mathfrak{m}}$ and $\text{Ad}_h$ are interchangeable: 
	\[
	p_{\mathfrak{m}}(\text{Ad}_h(\xi))=(\beta^{-1}\circ T_e\pi\circ\text{Ad}_h)(\xi)=(\beta^{-1}\circ T_o\tau\circ T_e\pi)  (\xi)=(\text{Ad}_h\circ \beta^{-1}\circ T_e\pi)(\xi)=\text{Ad}_h(p_{\mathfrak{m}}(\xi)).
	\]
	Using this, we have
	\begin{eqnarray}
	\nonumber {\bar \F}_2(\text{Ad}_h(\zeta))&=&\inf \{(\F_1)_e(\text{Ad}_h(\xi))~|~p_{\mathfrak{m}}(\xi)=\zeta \} \\
	\nonumber &=&\inf \{(\F_1)_e(\xi)~|~p_{\mathfrak{m}}(\xi)=\zeta \}=\bar{\F}_2(\zeta),
	\end{eqnarray}
	where in the last step we used that $(\F_1)_e$ is $\text{Ad}(H)$-invariant.

Since the norm ${\bar F}_2$ is subduced, it defines a nonlinear lift $\bar\h: {\mathfrak m} \to {\mathfrak g}$. This lift satisfies $(\F_1)_e(\bar \h(\zeta)) = {\bar \F}_2 (\zeta)$. It has the required equivariance because of
\[
(\F_1)_e(\bar \h(\text{Ad}_h\zeta)) = {\bar \F}_2 (\text{Ad}_h\zeta) = \bar{\F}_2(\zeta) = (\F_1)_e(\bar \h(\zeta)) = (\F_1)_e (\text{Ad}_h \bar\h(\zeta)).
\]
Since both $p_{\mathfrak m}(\bar\h(\text{Ad}_h\zeta)) = \text{Ad}_h\zeta$ and $p_{\mathfrak m}(\text{Ad}_h \bar\h(\zeta)) = \text{Ad}_h(p_{\mathfrak m} ( \bar\h(\zeta))= \text{Ad}_h\zeta$, we may conclude from the  uniqueness of the nonlinear lift that
\[
\bar\h \circ \text{Ad}_h = \text{Ad}_h \circ \bar\h.
\]

(ii)	We now extend the $Ad(H)$-invariant Minkowski norm ${\bar\F}_2$ on $\mathfrak m$ to a function $(\F_2)_n$ on $T_n(G/H)$: if $n=gH$ we set
	\[
	(\F_2)_n(v_n)={\bar\F}_2\left(\beta^{-1}(T_n\tau_{g^{-1}}(v_n))\right).
	\]
	This is	independent of the choice of $g$ because $T\tau_{(gh)^{-1}} =T\tau_{h^{-1}} \circ T\tau_{g^{-1}}$. The result then follows from relation (\ref{Ad2}) and the $\text{Ad}(H)$-invariance of ${\bar\F}_2$.
	
We may therefore conclude that $F_2: T(G/H)\to\R$, with 
\[
F_2(n,v_n) = (\F_2)_n(v_n),
\]	
is a Finsler function on $TN$. 

Consider $g\in G$ with $\pi(g) = n$. The Minkowski norm $(\F_1)_g$ on $T_gG$ is given by $(\F_1)_g(v_g) = F_1(g,v_g)$. In view of Definition~\ref{finslersub}, we only need to show that  
	\[
	T_g\pi :(T_gG,(\F_1)_g )\rightarrow (T_{n} (G/H),(\F_2)_n)
	\]
	is a submersion between Minkowski normed spaces for all $g\in G$ to be able to conclude that $\pi: (G,F_1) \to (G/H,F_2)$ is a submersion of Finsler manifolds. To show that, we prove that for any $v_n\in T_n(G/H)$,
	\begin{equation} \label{nogeenhulpke}
	(\F_2)_n(v_n) = \inf \{(\F_1)_g(u_g)~|~T_g\pi (u_g)=v_n\}.
	\end{equation}
	Using the definitions of $\F_2$ and $p_{\mathfrak{m}}$ and the left-invariance of $\F_1$ we get
	\begin{eqnarray*}
	\nonumber 	(\F_2)_n(v_n)&=&(\bar\F_2)_n((\beta^{-1}\circ T_n\tau_{g^{-1}})(v_n)) \\
	\nonumber &=& \inf \{(\F_1)_e(\xi)~|~p_{\mathfrak{m}}(\xi)=(\beta^{-1}\circ T_n\tau_{g^{-1}})(v_n)\} \\
	\nonumber &=&\inf \{(\F_1)_e(\xi)~|~ (\beta^{-1}\circ T_e\pi)(\xi)=(\beta^{-1}\circ T_n\tau_{g^{-1}})(v_n)\} \\
	\nonumber &=&\inf \{(\F_1)_e(\xi)~|~T_e\pi(\xi)= T_n\tau_{g^{-1}}(v_n)\} \\
	\nonumber &=&\inf \{(\F_1)_g(T_eL_g\xi)~|~T_e\pi(\xi)= T_n\tau_{g^{-1}}(v_n)\}.
	\end{eqnarray*}
Notice that $T_eL_g$ is bijective, and therefore an arbitrary element $u_g\in T_gG$ can be written in the form of $T_eL_g\xi$ with $\xi\in\mathfrak{g}$. On the other hand, relation (\ref{general}) implies that the equation $T_e\pi(\xi)= T_n\tau_{g^{-1}}(v_n)$ is equivalent with  $T_e\pi(u_g)=v_n$. This proves expression (\ref{nogeenhulpke}). 

(iii) Composing the defining relation of $\bar h$ with $(\beta^{-1}\circ T_n\tau_{g^{-1}})(v_n)$ leads to
\[
((\F_1)_e\circ \bar\h )\left( (\beta^{-1}\circ T_n\tau_{g^{-1}})(v_n) \right) ={\bar\F}_2\left(\beta^{-1}(T_n\tau_{g^{-1}}(v_n))\right) = (\F_2)_n(v_n).
\]
On the other hand, by the definition of $h$
\[
(\F_1)_e ( h(e,v_n))=(\F_2)_n(v_n).
\]
From the uniqueness of the nonlinear lifts, we conclude that 
\[
h(e,v_n) = \bar \h \left( (\beta^{-1}\circ T_n\tau_{g^{-1}})(v_n) \right).
\]
In \cite{ourthirdarticle} (Proposition~7 and 8) we have shown that an invariant Lagrangian function induces an invariant associated splitting, in the sense that 
\[
h(g,v_n)=T_eL_g(h(e,v_n)).
\]
From that, we see that
\[
h(g,v_n)=T_eL_g\left( \bar \h \left( (\beta^{-1}( T_n\tau_{g^{-1}}(v_n)) \right)\right).
\]
\end{proof}
In the  statement of Proposition~\ref{mainn} we required that  $F_1$ was a  left $G$- and right $H$-invariant Finsler function on $G$. In the next section we will 
construct examples where we start from a Finsler function that satisfies this property, namely a bi-invariant Finsler function on $G$.

\section{Subduced Finsler functions on $S^3$}  \label{exsec}

In what follows we focus our attention to semisimple Lie algebras $\la$ and discuss the relation between $\text{Ad}$ and $\text{Ad}^*$-invariant functions from Finsler functions on $\la$ and $\la^*$. Later, we  apply these observations to construct new examples of Finsler functions on $S^3$, subduced from Finsler functions on $SO(4)$.

A Lie algebra $\mathfrak{g}$ is {\em semisimple} if the Killing form $\kappa: \mathfrak{g}\times\mathfrak{g}\rightarrow \R$, defined by
	\[
	\kappa (\xi,\eta)=\tr (\ad_\xi\circ \ad_\eta)
	\]
	is non-degenerate. The Killing form $\kappa$ is a symmetric bilinear form that is invariant under all automorphisms of $\mathfrak{g}$. In particular, $\kappa$ is $Ad_g$-invariant for all $g\in G$:
		\[
		\kappa (Ad_g\xi,Ad_g\eta )=\kappa (\xi,\eta).
		\]
The linear maps $\ad_\xi$ are skew-adjoint w.r.t $\kappa$ for all $X\in\mathfrak{g}$:
		\[
		\kappa (\ad_\xi(\eta),\zeta)=-\kappa (\eta,\ad_\xi(\zeta)).
		\]

It is well-known that a connected Lie-group $G$ is compact and semisimple if and only if the Killing form of its Lie-algebra is negative definite.

A function $\BL$ on ${\mathfrak g}$ is $\text{Ad}(G)$-invariant if $\BL(\text{Ad}_g \xi) = \BL(\xi)$, for all $g\in G$. Likewise,  
a function $\H$ on ${\mathfrak g}^*$ is $\text{Ad}^*(G)$-invariant if $\H(\text{Ad}^*_g \theta) = \H(\theta)$.

The Killing form defines a `flat' map $\xi \in \la \mapsto \xi^\flat \in \la^*$, where $\xi^{\flat}(\eta) = \kappa(\xi,\eta)$, for all $\eta\in\la$ and, likewise, an inverse `sharp' map $\theta \in \la^* \mapsto \theta^\sharp \in \la$.

\begin{proposition} \label{dual} On a semisimple Lie algebra $\la$, there exist a one-to-one correspondence between $\text{Ad}(G)$-invariant functions $\BL$ on $\la$
and $\text{Ad}^*(G)$-invariant functions $\H$ on $\la^*$,  given by
\[  \BL(\xi) = \H(\xi^\flat) \qquad \mbox{and} \qquad \H (\theta) = \BL(\theta^\sharp).
\]  
\end{proposition}
\begin{proof}
 It is clear that
\[
(\text{Ad}_g \xi)^\flat = \text{Ad}_{g}^* \xi^\flat
\]
since $(\text{Ad}_g \xi)^\flat (\eta)=\kappa(\text{Ad}_g \xi, \eta) = \kappa (\xi,\text{Ad}_{g^{-1}}\eta) = \xi^\flat(\text{Ad}_{g^{-1}} \eta) =   (\text{Ad}_{g}^* \xi^\flat)(\eta)$. Similarly
\[
(\text{Ad}^*_g \theta)^\sharp = \text{Ad}_{g} \theta^\sharp.
\]
With that, we see that, when $\H$ is $\text{Ad}^*(G)$-invariant, 
\[
\BL(\text{Ad}_g\xi) = \H((\text{Ad}_g\xi)^\flat) = \H(\text{Ad}_{g}^* \xi^\flat) = \H( \xi^\flat) = \BL(\xi),
\]
and similarly for the case when $\BL$ is assumed to be $\text{Ad}(G)$-invariant.
\end{proof}

It is of interest to look at some coordinate expressions. Let $\{E_a\}$ be a basis of $\la$, and $\{E^a\}$ the corresponding dual basis. We will use coordinates $w^a$ on $\la$ and $p_a$ on $\la^*$. If we set $[E_a,E_b] = C_{ab}^c E_c$, one may easily verify that a function $\BL(w)$ on $\la$ is invariant if 
\begin{equation} \label{biinvequations}
C^c_{ab}w^b\frac{\partial \BL}{\partial w^c}=0.
\end{equation}
One may think of this as a system of linear partial differential equations in $\BL$, and as such, a functionally independent maximal set, say $\BL_j(w)$, of $r$ solutions generate all the solutions, as their functional combinations: $\phi(\BL_1(w),\ldots,\BL_r(w))$. 

Likewise, a function $\H$ on $\la^*$ is invariant if
\[	C^c_{ab}p_c\frac{\partial \H}{\partial p_b}=0.
	\] 
In \cite{CHE} Chevalley proves that for semisimple Lie algebras of rank $r$, one may find a functionally independent set of invariant functions that consists of $r$ homogeneous polynomials $\H_i$, with degrees equal to the primitive exponents. 

We will now make the link to the previous discussions on Finsler submersions.
Recall that the set of bi-invariant (Finsler) functions on $TG$ is in one-to-one correspondence with the set of $\text{Ad}(G)$-invariant functions  on $\la$, and therefore also with the set of $\text{Ad}^*(G)$-invariant functions on $\la^*$.
We conclude:
\begin{proposition} \label{newchevy}
	Let $G$ be a connected semisimple Lie group with Lie algebra $\mathfrak{g}$ of rank $r$ and let $\H_i$ ($i=1,\ldots,r$) be functionally independent invariant polynomials on $\mathfrak{g}^*$, with corresponding polynomials $\BL_i=\H_i \circ .^\sharp$ on  $\la$. Then, the restriction of any bi-invariant function $\BL$ on $TG$ to the Lie algebra is of the form
	\[
	\BL(w)= \phi(\BL_1(w),\ldots,\BL_r(w)),
	\]
for a smooth function $\phi: \R^r \to \R$.
\end{proposition}

We now apply the results of the previous sections to the case of the Lie group $SO(4)$. We consider the subgroup $H$  of matrices of the form
\[
\begin{pmatrix}
	A       & 0 \\
	0       & 1
\end{pmatrix}
\]
with $A\in SO(3)$. Then, $H$ isomorphic to  $SO(3)$ and  it is well-known (see \cite{Lee}) that $SO(4)/SO(3) \cong S^3$. Let $\kappa$ be the Killing form and define $\mathfrak{m}$ as the orthogonal complement of $\mathfrak{so}(3)$ with respect to $-\kappa$. Then,  $\mathfrak{g}=\mathfrak{so}(4) =
\mathfrak{h}\oplus \mathfrak{m}$,
where $\mathfrak{h}\cong \mathfrak{so}(3)$ and where the subspace $\mathfrak{m}$ consist of $4\times4$ matrices of the form
\[
\begin{pmatrix}
0_3       & -x \\
x^t       & 0
\end{pmatrix}
\]
with $x$ a column vector in $\R^3$.

The set of matrices, given by
\[
E_1=
\begin{pmatrix}
0       & 0 & 0 & 0\\
0       & 0 & -1 & 0\\
0       & 1 & 0 & 0\\
0       & 0 & 0 & 0
\end{pmatrix},~
E_2=
\begin{pmatrix}
0       & 0 & 1 & 0\\
0       & 0 & 0 & 0\\
-1       & 0 & 0 & 0\\
0       & 0 & 0 & 0
\end{pmatrix},~
E_3=
\begin{pmatrix}
0       & -1 & 0 & 0\\
1       & 0 & 0 & 0\\
0       & 0 & 0 & 0\\
0       & 0 & 0 & 0
\end{pmatrix},~
\]
\\
\[
E_4=
\begin{pmatrix}
0       & 0 & 0 & -1\\
0       & 0 & 0 & 0\\
0       & 0 & 0 & 0\\
1       & 0 & 0 & 0
\end{pmatrix},~
E_5=
\begin{pmatrix}
0       & 0 & 0 & 0\\
0       & 0 & 0 & -1\\
0       & 0 & 0 & 0\\
0       & 1 & 0 & 0
\end{pmatrix},~
E_6=
\begin{pmatrix}
0       & 0 & 0 & 0\\
0       & 0 & 0 & 0\\
0       & 0 & 0 & -1\\
0       & 0 & 1 & 0
\end{pmatrix}
\]
constitute a basis of $\mathfrak{so}(4)$, where $\text{span}\{E_1,E_2,E_3\}\cong \mathfrak{so}(3)$ and $-\kappa(E_a,E_b)=\delta_{ab}$. The above basis is, however, not the standard basis. 
The standard basis respects the decomposition $\mathfrak{so}(4)\cong\mathfrak{so}(3)\oplus \mathfrak{so}(3)$ and is given by
\begin{eqnarray}\label{coordinatechange}
X_1=\frac{1}{2}(E_1+E_4),~X_2=\frac{1}{2}(E_2+E_5)~,~X_3=\frac{1}{2}(E_3+E_6), \\
\nonumber X_4=\frac{1}{2}(E_1-E_4),~X_5=\frac{1}{2}(E_2-E_5),~X_6=\frac{1}{2}(E_3-E_6).
\end{eqnarray}
Here, both $\text{span}\{X_1,X_2,X_3\}\cong \mathfrak{so}(3)$ and $\text{span}\{X_4,X_5,X_6\} \cong \mathfrak{so}(3)$, while ${\mathfrak m}=\text{span}\{E_4,E_5,E_6\}\not\cong \mathfrak{so}(3)$. 
We will use in what follows the  basis $\{E_a\}$ and its corresponding coordinates $w^a$, because in this way $\mathfrak{m}$ is an $\ad$-invariant subspace. The non-zero Lie brackets of the elements of the new basis are
\[
[E_i,E_j] = \epsilon_{ijk}E_k,\qquad [E_i,E_{j+3}] = \epsilon_{ijk}E_{k+3} \qquad\mbox{and}\qquad [E_{i+3},E_{j+3}] = \epsilon_{ijk}E_k.
\]

We know that $SO(3)$ has rank 1. In \cite{patera} it is shown that (if $p_a$ are coordinates w.r.t.\ the dual of the standard basis $\{X_1,X_2,X_3\}$ of ${\mathfrak{so}}(3)$) the function $p_1^2+p_2^2+p_3^2$ is an invariant function on ${\mathfrak{so}}^*(3)$. Since we know that $SO(4)$ has rank 2, and since the polynomials $\H_1=p_1^2+p_2^2+p_3^2$ and $\H_2=p_4^2+p_5^2+p_6^2$ are functionally independent, we know from Proposition~\ref{newchevy} that their corresponding functions $\BL_1 = \H_1 \circ .^\flat$ and $\BL_2 = \H_2 \circ .^\flat$ can be used to determine all other invariant functions as functional combinations of these two. In the current basis these correspond to the functions 
\begin{eqnarray*}
\BL_1(w)&=&(w^1+w^4)^2+(w^1+w^5)^2+(w^3+w^6)^2, \\ \BL_2(w)&=&(w^1-w^4)^2+(w^1-w^5)^2+(w^3-w^6)^2.
\end{eqnarray*}
on $\mathfrak{so}(4)$. One can also check directly that $\BL_1$ and $\BL_2$ do indeed satisfy the PDE (\ref{biinvequations}). 

Among the solutions of (\ref{biinvequations}),  we would like to find those that correspond to Minkowski norms (and therefore to bi-invariant Finsler functions on $TG$), i.e.\ those  that are positive, regular, 1-homogeneous and strongly convex. We claim that e.g.\ the following 3 functions satisfy the definition of a Minkowski norm:
\begin{eqnarray}
\nonumber	\F(w)&=&\sqrt[4]{(\BL_1(w)+\BL_2(w))^2}, \\
\nonumber	\hat\F(w)&=&\sqrt[4]{\BL_1^2(w)+\BL_2^2(w)},\\
\nonumber	\tilde\F(w)&=&\sqrt[4]{(\BL_1(w)+\BL_2(w))^2+\BL_1^2(w)}~=~\sqrt[4]{4\left(~\sum (w^i)^2~\right)^2 +\BL_1^2(w)}.
\end{eqnarray}
The first function, $\F$, is a negative constant multiple of the Killing form, and as such, it is an $\text{Ad}(G)$-invariant Euclidean norm on $\mathfrak{g}$. We show next that the functions $\hat\F$ and $\tilde \F$ are non-Euclidean Minkowski norms on $\mathfrak{g}$. To see that, we only need to prove that e.g.\ $\hat\E=\frac{1}{2}{\hat\F}^2$ is stricly convex (and a similar reasoning will hold for $\tilde \E = \frac{1}{2}\tilde\F^2$). This will follow if we can verify that, for any $x_0,y_0\in \mathfrak{so}(4)$, with $x_0\neq y_0$ and $y_0\neq0$, the function
 $\hat\E=\frac{1}{2}{\hat\F}^2$ is convex along the line $x_0+ty_0$. If we introduce
\[
 \varphi(u_1,u_2)=\frac{1}{2}\sqrt{u_1^2+u_2^2}\qquad\text{and}\qquad 
\psi_i(t)=\BL_i(x_0+ty_0),
\]
then
\[
\Big({\hat\E}(x_0+ty_0)\Big)'' = 	(\psi_i'(t))^T.\Big({\rm Hess\,}\varphi(\psi_i(t))\Big).\psi_i'(t) + (\nabla\varphi(\psi_i(t)))^T.\psi_i''(t).
\]
The first term in this sum is non-negative because the Hessian matrix of $\varphi$ is positive semi-definite (since $\varphi$ is convex).   The gradient of $\varphi$ is the vector $\frac{1}{4\sqrt{u_1^2+u_2^2}}(u_1,u_2)$. When evaluated along  $\psi_i(t)\geq 0$, both components are non-negative. Moreover, the second derivates of $\psi_i$ are strictly positive if $y_0\neq 0$. For this reason, also the second term is positive and therefore the function ${\hat \E}$ is strictly convex.

We conclude that $\hat\F$ and $\tilde\F$ are Minkowski norms on $\mathfrak{so}(4)$. Since they are invariant, we know that they correspond to bi-invariant  Finsler functions $\hat F$ and $\tilde F$ on $SO(4)$. The bi-invariance is sufficient to ensure that  the conditions of  Proposition~\ref{mainn} are satisfied, for $\pi: SO(4) \to SO(4)/SO(3) \cong S^3$. As a consequence they define subduced Finsler functions ${\hat F}_2$ and ${\tilde F}_2$ on $S^3$. 

We end this paragraph with an expression for the corresponding subduced Minkowski norms $({\hat\F}_2)_o$ on $T_o(G/H) \cong {\mathfrak g}/{\mathfrak h}$, which follow from the Minkowski submersion $T_e\pi=\mu$. After taking the isomorphism $\beta$ into account we may identify this space with ${\mathfrak m} \cong \mathfrak{so}(3)$, and the map $\beta\circ\mu$
with the projection to the last three coordinates. The subduced Minkowski norm $({\hat\F}_2)_o$ (regarded as a function on ${\mathfrak m}$) can be characterized as
\begin{equation*}
({\hat\F}_2)_o(w_0^4,w_0^5,w_0^6)~=~\min_{(w^1,w^2,w^3) \in\R^3} \{~ {\hat\F}(w^1,w^2,w^3,w_0^4,w_0^5,w_0^6) ~\},
\end{equation*}
and likewise for $\F$ and $\tilde\F$.

The vectors in $\la$, at which the minimum is attained form the horizontal cone w.r.t.\ the submersion $\mu=T_e\pi$ of Minkowski spaces. For $\F$ the Minkowski submersion is an Euclidean submersion, therefore the horizontal cone is a subspace, orthogonal to $\mathfrak{m}$, which is by construction $\mathfrak{h}$. 

It is interesting that the horizontal cone corresponding to $\hat\F$ also coincides with the subspace $\mathfrak{h}$. This can be seen as follows. For  fixed values of $(w_0^4,w_0^5,w_0^6)$, we can see - after a lengthy calculation - that the function 
\[
\sqrt[4]{((w^1+w_0^4)^2+(w^2+w_0^5)^2+(w^3+w_0^6)^2)^2 +((w^1-w_0^4)^2+(w^2-w_0^5)^2+(w^3-w_0^6)^2)^2}
\] 
admits only one critical point at $(w^1,w^2,w^3) = (0,0,0)$, which is a local minimum. Since the function is convex, it is also a global minimum, whence our claim. In this case the induced nonlinear lift from $\mathfrak m$ to $\mathfrak g$ is simply
\[
(w_0^4,w_0^5,w_0^6) \mapsto (0,0,0,w_0^4,w_0^5,w_0^6).
\]
We conclude that $\hat\F$ is a non-Riemannian Finsler function, whose induced nonlinear splitting is in fact linear. 

On the other hand, the subduced Minkowski norm corresponding to $\tilde\F$ admits a proper (nonlinear) cone in $\mathfrak{so}(4)$. One can check that the horizontal lift of the vector $(w_0^4,w_0^5,w_0^6)=(1,0,0)$ is $(0,0,0,1,0,0)$, but the horizontal lift of $(1,1,0)$ is not simply $(0,0,0,1,1,0)$.

\end{document}